\documentclass[10pt]{article}
\usepackage{amsmath}
\usepackage{amsfonts}
\usepackage{amssymb}
\usepackage{amsthm}
\usepackage{palatino}
\usepackage[margin=2.5cm, vmargin={1.5cm},includefoot]{geometry}

\title{Kernels of $L$-functions of cusp forms}

\author{Nikolaos Diamantis and Cormac O'Sullivan}

\begin{document}

\maketitle

\def\s#1#2{\langle #1 , #2 \rangle}

\def\H{{\mathbb H}}
\def\F{{\frak F}}
\def\C{{\mathbb C}}
\def\R{{\mathbb R}}
\def\Z{{\mathbb Z}}
\def\Q{{\mathbb Q}}
\def\N{{\mathbb N}}
\def\B{{\mathbb B}}
\def\G{{\Gamma}}
\def\GH{{\G \backslash \H}}
\def\g{{\gamma}}
\def\L{{\Lambda}}
\def\ee{{\varepsilon}}
\def\K{{\mathcal K}}
\def\Re{\mathrm{Re}}
\def\Im{\mathrm{Im}}
\def\PSL{\mathrm{PSL}}
\def\SL{\mathrm{SL}}
\def\Vol{\operatorname{Vol}}
\def\li{\operatorname{Li}}
\def\sgn{\operatorname{sgn}}
\def\coh{{\mathcal C}}
\def\doh{{\mathcal D}}
\def\lqs{\leqslant}
\def\gqs{\geqslant}
\def\ab{\mathring}

\def\ca{{\frak a}}
\def\cb{{\frak b}}
\def\cc{{\frak c}}
\def\cd{{\frak d}}
\def\ci{{\infty}}

\def\sa{{\sigma_\frak a}}
\def\sb{{\sigma_\frak b}}
\def\sc{{\sigma_\frak c}}
\def\sd{{\sigma_\frak d}}
\def\si{{\sigma_\infty}}


\newtheorem{theorem}{Theorem}[section]
\newtheorem{lemma}[theorem]{Lemma}
\newtheorem{prop}[theorem]{Proposition}
\newtheorem{cor}[theorem]{Corollary}

\renewcommand{\labelenumi}{(\roman{enumi})}

\numberwithin{equation}{section}

\bibliographystyle{plain}

\begin{abstract}\noindent
We give a new expression for the inner product of two kernel functions associated to a cusp form. Among other applications, it yields an extension of a formula of Kohnen and Zagier, and another proof of Manin's Periods Theorem. Cohen's representation of these kernels as series is also generalized.
\end{abstract}

\section{Introduction}
\subsection{Background} Let
\begin{equation}\label{fourierexp}
f(z)=\sum_{n=1}^\infty a_f(n)  e^{2 \pi i n z}
\end{equation}
be  in $S_{k}(\Gamma)$, the $\C$-vector space of holomorphic, weight $k$ cusp forms for the modular group $\G=\PSL_2(\Z)$.
 The $L$-function of $f$ is
\begin{equation}\label{lfs}
L(f,s) := \sum_{n=1}^\infty \frac{a_f(n)}{n^s}
\end{equation}
defined for $\Re(s)$ large. It is an Euler product when $f$ is an eigenfunction of all Hecke operators $T_m$. Let $\mathcal B_k$ be the unique basis of $S_{k}$ consisting of such Hecke eigenforms, normalized to have $a_f(1)=1$. The completed $L$-function is
\begin{equation}\label{lfs2}
L^*(f,s):=(2\pi)^{-s}\G(s) L(f,s) = \int_0^\infty f(iy) y^{s-1} \, dy
\end{equation}
and is analytic for all $s \in \C$.
 For integers $n$ with $0 \lqs n \lqs k-2$ the $n$th  period of $f$ is
$$
r_n(f):=L^*(f,n+1).
$$
A celebrated result of Manin, his Periods Theorem \cite{Ma1}, states that the ratios of all the periods for $n$ even (and separately for $n$ odd) lie in the field $K_f$ generated by the coefficients $a_f(n)$ when $f \in \mathcal B_k$. His proof uses the Eichler-Shimura isomorphism and a computation involving continued fractions. Shimura extends Manin's result to all Hecke congruence groups with a different proof \cite{Sh}.
Zagier in \cite[$\S 5$]{Za3}  provides another route to the Periods Theorem. This proof  relies on the Rankin-Cohen bracket (\ref{rcb}) and extending an identity of Rankin (\ref{zaggg}). We give a new proof of Manin's Periods Theorem in section \ref{manin} by extending a result of Kohnen and Zagier in \cite{KZ} which we describe next.
With the Petersson inner product
\begin{equation}\label{pip}
 \s{f}{g} :=\int_{\GH} y^{k} f(z) \overline{g(z)} \, d\mu z
\end{equation}
 there must exist $R_n \in S_k$ such that
 \begin{equation}\label{sfrn}
\s{f}{R_n}=r_n(f)
\end{equation}
  for all $f \in S_k(\G)$ and every $0 \lqs n \lqs k-2$.
Kohnen and Zagier show that, remarkably, for $m \not\equiv n \mod 2$, $\s{R_m}{R_n}$ is a rational number given by an explicit formula involving the Bernoulli numbers. To state it, for $n \in \Z$ put
\begin{equation}\label{rho}
\rho(2n):=\begin{cases} (-1)^{n+1} B_{2n}/(2n)!&  n \gqs 0 \\
 0 & n<0 \end{cases}
\end{equation}
so that
$\rho(0)=-1$ and $\rho(2n)>0$ for $n>0$.
Set $\tilde m:=k-2-m$ and  $\tilde n:=k-2-n$.
\begin{theorem} \cite{KZ} \label{rmrnsimple}
For integers $m,n$ of opposite parity with $0 < m,n < k-2$
\begin{eqnarray*}
2^{2-k}(k-2)!\bigl \langle R_m, R_n \bigr \rangle & = &
 \rho(m - \tilde n + 1) m! n!
 +
 \rho(-m +  \tilde n + 1)  \tilde m! \tilde n! \nonumber
\\ &  & +
(-1)^{k/2} \rho(m -  n + 1) m! \tilde n!
  +
(-1)^{k/2} \rho(-m + n + 1)  \tilde m! n! .
\end{eqnarray*}
\end{theorem}
For simplicity we have omitted the cases when $m$ or $n$ equals $0$ or $k-2$. See Theorem \ref{rmrn} for the complete statement.

\subsection{Statement of main results}
We further this study to non-critical values by focusing on
the kernel function of $L^*(f, s)$ rather than $L^*(f, n)$ with $n$ a
critical value only. One of our motivating questions was to what extent formulas, such as that for $\s{R_m}{R_n}$ generalize. Indeed, extending (\ref{sfrn}), for every $s \in \C$ there must exist $\doh_k(z,s) \in S_k$ such that
\begin{equation}\label{doh}
\s{\doh_k(\cdot,s)}{f}=L^*(\overline{f},s)
\end{equation}
for all $f \in S_k$. Clearly, $R_n = \doh_k(\cdot,n+1)$. Our first main result shows that the Petersson scalar product of two
values of the a priori unknown kernel $\mathcal D_k(\cdot, s)$ can be
explicitly interpreted in terms of familiar objects.
With the Poincar\'e series
\begin{equation}\label{poincare}
P_l(z):=\sum_{\g \in \G_\infty \backslash \G} \frac {e^{2\pi i l z}} {
   j(\g,z)^k} \ \ \ \in S_k
\end{equation}
for $j((\smallmatrix a & b \\ c & d \endsmallmatrix), z):=cz+d$  and  non-holomorphic Eisenstein series
 \begin{equation*}
E^*_k(z,s):= \pi^{-s} \G(s+|k|/2) \zeta(2s) \sum_{\g \in \G_\ci \backslash \G} \Im (\g z)^s \left(\frac{j(\g, z)}{|j(\g, z)|}\right)^{-k}
\end{equation*}
 we prove the following.
\begin{theorem} \label{dsdw}
For all $s,w \in \C$
\begin{equation} 2 \cdot \pi^{k/2} \G(k-1) \Bigl\langle T_l \doh_k(\cdot,s), \doh_k(\cdot,\overline w)\Bigr\rangle =
(-1)^{k_2/2} (4\pi l)^{k-1}
\Bigl\langle P_l, y^{-k/2}E^*_{k_1}(\cdot, \bar u)E^*_{k_2}(\cdot, \bar v)\Bigr\rangle.
\label{deqn}
\end{equation}
Here  $k_1,k_2$ are any non-negative even integers with $k_1+k_2=k$ and
\begin{equation}\label{uvsw}
    2u=s+w-k+1,\quad 2v=-s+w+1.
\end{equation}
\end{theorem}
Including the operator $T_l$ on the left of (\ref{deqn}) is very natural,  giving  a description of the Hecke action and, as we shall show in
section \ref{arit}, there are interesting arithmetic applications.
We use Theorem \ref{dsdw} to prove an extension of Theorem \ref{rmrnsimple}.
 With $s,w \in \Z_{\gqs 0}$  set
\begin{equation}\label{zdef}
Z_{s,w}(x):=
  (-1)^{(s+w+1)/2}  \binom{k-2}{s-1}^{-1}  \sum_{r=0}^{k-1} (-x)^r \binom{k-1-w}{r} \binom{w-1+r}{k-1-s}.
\end{equation}
\begin{theorem} \label{rmrn2}
For $4 \lqs k \in 2\Z$ and integers $s,w$ of opposite parity satisfying $1\lqs s,w \lqs k-1$
\begin{multline}
  (k-2)! 2^{2-k}\Bigl\langle T_l \doh_k(\cdot,s), \doh_k(\cdot, w)\Bigr\rangle  =  \sigma_{2v-1}(l)
  \Bigl[\rho(2u)l^{k-1-w}  \G(s)\G(w)
     +
     \rho(2-2u)l^{s-1} \G(k-s)\G(k-w)\Bigr] \\
   +  (-1)^{k/2} \sigma_{2u-1}(l)
  \Bigl[ \rho(2v)l^{k-1-w}  \G(k-s)\G(w)
     +
     \rho(2-2v)l^{k-1-s} \G(s)\G(k-w)\Bigr] \\
     +
     2 (-1)^{k/2} (k-2)! l^{k-1-w} \sum_{n=1}^{l-1} \sigma_{2u-1}(n)\sigma_{2v-1}(l-n)  Z_{s,w}(n/l)\\
   - \frac{\sigma_{k-1}(l)}{(k-1) \rho(k)} \Bigl[ \left(\delta_{w,1}(-1)^{(k-s)/2} + \delta_{w,k-1}(-1)^{s/2}\right)\G(s)\G(k-s)\rho(s)\rho(k-s) \\
+ \left(\delta_{s,1}(-1)^{(k-w)/2} + \delta_{s,k-1}(-1)^{w/2}\right)\G(w)\G(k-w)\rho(w)\rho(k-w)\Bigr]. \label{bigformula}
\end{multline}
\end{theorem}
For $s=m+1$, $w=n+1$, this gives an explicit  expression for the rational number $\s{T_l R_m}{R_n}$.
As described below, statements equivalent to Theorem \ref{rmrn2} have already appeared in the literature, using a variety of proof techniques.
 Our proof is new
 and relies on
  choosing  $k_1$ and $k_2$ so that the series $E^*_{k_1}(z, \bar u)$ and $E^*_{k_2}(z, \bar v)$ above only have terms in their Fourier expansions with $e^{2\pi i nx}$ for $n\gqs 0$. In this way we obtain finite sums from the right side of (\ref{deqn}). This is carried out in sections \ref{eisens} and \ref{three}. In section \ref{manin},  as an application of Theorem \ref{rmrn2}, we prove Manin's Periods Theorem with similar methods to those of \cite{Sh, Za3, KZ}.

\vskip 3mm
The advantage of using Theorem \ref{dsdw} is that it allows us  to consider $\doh_k(z,s)$, $\doh_k(z, w)$ and the inner product $\bigl\langle T_l \doh_k(\cdot,s), \doh_k(\cdot, w)\bigr\rangle$ at other interesting values of $s$ and $w$. This is the subject of
current work. For example,  $\s{R_n}{R_n}$ is related by (\ref{deqn}) to Eisenstein series with $u,v$ half-integral. Further, since both sides of (\ref{deqn}) are analytic in $s$ and $w$ we may study derivatives of $L$-series. Finally, in relation to Theorem \ref{dsdw}, we speculate that it might be used to uncover weaker forms of the Periods Theorem for values outside the critical strip.

\vskip 3mm
Kohnen and Zagier's first  proof of Theorem \ref{rmrnsimple} in \cite[Theorem 1, p 203]{KZ} uses a  holomorphic kernel due to Cohen:
\begin{equation}\label{cohen}
\coh_k(z,s):=\sum_{\g \in \G} \frac 1 {(\g z)^{s} j(\g,z)^k}
\end{equation}
with $z$ in the upper half plane $\H$ and $s$ taking integer values
between $2$ and $k-2$. Up to a constant, (\ref{cohen}) is $R_{s-1}$. As with $\doh_k(z,s)$, we may examine $\coh_k(z,s)$ as $s=\sigma+it$ ranges over all of $\C$.  With $z
\in \H$ and $s \in \C$, the expression $z^s$ is well defined by
\begin{equation}\label{zs}
z^s=e^{s \log z},
\end{equation}
where we take the principal branch of the log. For any fixed $s \in \C$,
$z^s$ is a holomorphic function of $z$ in $\H$. We prove the following result.

\begin{theorem}
The series $\coh_k(z,s)$ defined by (\ref{cohen}) is absolutely convergent for $\sigma \in (1,k-1)$. The convergence is uniform for $\sigma$ in compact subsets of $(1,k-1)$. For each $s$ with $\sigma \in (1,k-1)$ we have $\coh_k(z,s) \in S_k(\G)$, the space of holomorphic, weight $k$ cusp forms for $\G$.
\end{theorem}

This is proved in  section \ref{cgt} where $\coh_k(z,s)$ is better understood as a special case of the series
$$
\Omega_\ci(z,\tau;s,k):= \sum_{\g \in \G} \frac 1 {(\g z-\overline{\tau})^{s} j(\g,z)^k}
$$
with $z \in \H$, $\tau \in \H \cup \R$ and $\G$ a Fuchsian group of the first kind. We show in (\ref{connect}) that
\begin{equation*}
\coh_k(z,s) = 2^{2-k}(-1)^{k/2} \pi e^{-s i \pi/2}  \frac{\G(k-1)}{\G(s)\G(k-s)} \doh_k(z,s).
\end{equation*}
To our knowledge, this is the first explicit construction of a kernel of
$L^*(f, s)$
for $s$ in the critical strip in terms of a series. In most of the many
works in which explicit kernels play an important role, for example \cite{GoZ, GZ,
Sh, St}, what is expressed as an inner product is not $L^*(f, s)$
itself but either the critical values of $L^*(f, s)$ or products of
different values of $L^*(f, s)$. It is natural to expect that expressions
of
$L^*(f, s)$ itself as an inner product will be easier to handle,
especially in
questions involving analytic aspects such as derivatives of $L$-functions
etc.

\vskip 3mm
Kohnen and Zagier's second proof of Theorem \ref{rmrnsimple} in \cite[Theorem 3, p 215]{KZ} employs the Rankin-Cohen bracket and special cases of
Theorem \ref{rmrn2} appear in \cite[p 215]{KZ}. This proof is extended and elegantly recast by Zagier \cite{Za4} into a four variable generating function.
Theorem \ref{rmrn2} may be derived, with some work, from \cite[Main Theorem (3rd version), p 460]{Za4}. Fukuhara  exploits Dedekind symbols to prove  a result equivalent to Theorem \ref{rmrn2} in \cite[Theorem 2.9 (2)]{Fu}.
These theorems have also been extended to higher levels by a number of authors, with proofs using the Cohen kernel.
Antoniadis in \cite{An} and Fukuhara and Yang in \cite{FY} generalize  Theorem \ref{rmrnsimple} to all Hecke congruence groups $\G_0(N)$. The results of \cite{An} are valid for $N$ square free. Much simpler formulas,  valid for all $N$, are found in \cite[Theorem 1.1]{FY}. In
\cite[Theorem 1.3]{FY} they further show how to extend Theorem \ref{rmrn2} to $\G_0(N)$. We expect our methods to extend naturally to these higher levels too.

\vskip 3mm
{\bf Acknowledgement.} We thank Professor Fukuhara for his valuable comments on an earlier version of this paper.

\section{Eisenstein  series and the kernel $\doh_k$} \label{eisens}
\subsection{An inner product formula}
We recall some properties of non-holomorphic Eisenstein series  needed in the sequel.
Set
$$
\theta_k(s) :=  \pi^{-s} \G(s+|k|/2) \zeta(2s)
$$
and for a convenient normalization put
\begin{equation}\label{defe}
E^*_k(z,s):= \theta_k(s) \sum_{\g \in \G_\ci \backslash \G} \Im (\g z)^s \left(\frac{j(\g, z)}{|j(\g, z)|}\right)^{-k}.
\end{equation}
Then (\ref{defe}) converges to an analytic function of $s\in \C$ and $z\in \H$ for $\Re(s)>1$. It transforms as
$$
E^*_k(\g z,s) =  \left(\frac{j(\g, z)}{|j(\g, z)|}\right)^{k} E^*_k(z,s)
$$
for all $\g \in \G$.
The weight $0$ Eisenstein series has the Fourier expansion
\begin{equation}\label{eisexp}
E^*_{0}(z;s) = \theta(s) y^s+ \theta(1-s) y^{1-s} + \sum_{0 \neq m \in \Z} \frac{ \sigma_{2s-1}(|m|)}{|m|^s} W_s(m z)
\end{equation}
as shown in \cite[Theorem 3.4]{Iwsp} where $W_s$ is the Whittaker function and
\begin{equation}\label{sigma}
\sigma_s(m):=\sum_{d|m} d^s = m^s \sigma_{-s}(m)
\end{equation}
the usual divisor function.
With the weight lowering and raising operators
\begin{equation}\label{wlr}
L_k := -2iy \frac{d}{d\overline{z}} -k/2, \quad R_k := 2iy \frac{d}{dz} +k/2
\end{equation}
we have
\begin{eqnarray}
  L_k E^*_k(z,s) &=& \begin{cases} E^*_{k-2}(z,s) \quad & k \lqs 0 \\
 (s+|k|/2-1)(s-|k|/2) E^*_{k-2}(z,s) \quad & k > 0 \end{cases}, \label{lower} \\
  R_k E^*_k(z,s) &=& \begin{cases} E^*_{k+2}(z,s) \quad & k \gqs 0 \\
 (s+|k|/2-1)(s-|k|/2) E^*_{k+2}(z,s) \quad & k < 0 \end{cases}. \label{raise}
\end{eqnarray}
Thus, for $k \in 2\Z$,
\begin{equation*}
E_k^*(z, s)=\theta_k(s)y^s+\theta_k(1-s)y^{1-s}+
\sum_{0 \neq l \in \Z}\frac{\sigma_{2s-1}(|l|)}{|l|^s}
\sum_{r=-k/2}^{k/2} \mathcal P_r^{k/2}(-4\pi ly)W_{s+r}(lz)
\end{equation*}
where $\mathcal P_r^{k/2}$ is a polynomial of degree $k/2$ that may be given explicitly \cite{Oeis}.
Hence $E_k^*(z, s)$ has a meromorphic continuation to all $s \in \C$.

\vskip 3mm
Recall from the introduction that $\mathcal B_k $ is the basis for $S_k$ of Hecke eigenforms, normalized with first coefficient $1$. Thus, for any $f \in \mathcal B_k$ we have $T_l f = \lambda_f(l) f$ with $f(z)=\sum_{l=1}^\infty \lambda_f(l) e^{2\pi i lz}$. Also  $\lambda_f(l) \in \R$ since $\s{T_l f}{f} = \s{f}{T_l f}$.
We will need the next formula.
\begin{prop} \label{iprodprop}
Let $k_1$, $k_2$ be even and non-negative with $k=k_1+k_2$. Then for $f \in \mathcal B_k$ and all $s,w \in \C$
\begin{equation}\label{ziggg}
 2 \cdot \pi^{k/2} L^*(f,s) L^*(f,w) = (-1)^{k_2/2} \Bigl\langle f, y^{-k/2}E^*_{k_1}(\cdot, \bar u)E^*_{k_2}(\cdot, \bar v)\Bigr\rangle
\end{equation}
with $u$ and $v$ given by (\ref{uvsw}).
\end{prop}
\begin{proof}
Define the convolution $L$-series
$$
L(f\otimes E(\cdot,v),w):=\sum_{n=1}^\infty \frac{a_f(n) \sigma_{2v-1}(n)}{n^w}.
$$
Unfolding $E^*_{k}(z, \bar u)$ we find
\begin{equation}\label{convv}
\Bigl\langle f, y^{-k/2}E^*_{k}(\cdot, \bar u)E^*_{0}(\cdot, \bar v)\Bigr\rangle =
\frac{\zeta(2u) \G(s) \G(w)}{2^{k-2+2u} \pi^{k/2-1+2u}} L(f\otimes E(\cdot,v),w).
\end{equation}
\begin{lemma} \label{label}
For even $k_1,k_2$ with $k=k_1+k_2$ and $k_1, k_2-2 \gqs 0$
$$
\Bigl\langle f, y^{-k/2}E^*_{k_1}(\cdot, \bar u)E^*_{k_2}(\cdot, \bar v)\Bigr\rangle =
-\Bigl\langle f, y^{-k/2}E^*_{k_1+2}(\cdot, \bar u)E^*_{k_2-2}(\cdot, \bar v)\Bigr\rangle.
$$
\end{lemma}
\begin{proof}
With $\s{}{}_0$ denoting the inner product (\ref{pip}) with $k=0$,
\begin{eqnarray*}
  \Bigl\langle y^{k/2} f, E^*_{k_1}(\cdot, \bar u)E^*_{k_2}(\cdot, \bar v)\Bigr\rangle_0 &=&
  \Bigl\langle y^{k/2} f E^*_{-k_1}(\cdot, u), R_{k_2-2} E^*_{k_2-2}(\cdot, \bar v)\Bigr\rangle_0 \\
   &=& -\Bigl\langle L_{k_2}\left( y^{k/2} f E^*_{-k_1}(\cdot, u)\right),  E^*_{k_2-2}(\cdot, \bar v)\Bigr\rangle_0 \\
   &=& -\Bigl\langle L_{k}\left( y^{k/2} f \right) E^*_{-k_1}(\cdot, u) + y^{k/2} f L_{-k_1}\left( E^*_{-k_1}(\cdot, u)\right),  E^*_{k_2-2}(\cdot, \bar v)\Bigr\rangle_0  \\
   &=& -\Bigl\langle  y^{k/2} f E^*_{-k_1-2}(\cdot, u),  E^*_{k_2-2}(\cdot, \bar v)\Bigr\rangle_0 \\
   &=& -\Bigl\langle  y^{k/2} f , E^*_{k_1+2}(\cdot, \bar u) E^*_{k_2-2}(\cdot, \bar v)\Bigr\rangle_0.
\end{eqnarray*}
We used (\ref{lower}), (\ref{raise}) and that $L_{k}\left( y^{k/2} f \right)=0$. Moving the lowering and raising operators inside the inner product is justified in \cite[Prop. 9.3]{JO1}, for example. The lemma is proved.
\end{proof}
It follows easily that
\begin{equation}\label{ekip}
\Bigl\langle f, y^{-k/2}E^*_{k_1}(\cdot, \bar u)E^*_{k_2}(\cdot, \bar v)\Bigr\rangle = (-1)^{k_2/2} \Bigl\langle f, y^{-k/2}E^*_{k}(\cdot, \bar u)E^*_{0}(\cdot, \bar v)\Bigr\rangle
\end{equation}
for $k=k_1+k_2$ and $k_1, k_2 \gqs 0$.
Combining (\ref{convv}) and (\ref{ekip}) shows
\begin{equation}\label{ll}
\Bigl\langle f, y^{-k/2}E^*_{k_1}(\cdot, \bar u)E^*_{k_2}(\cdot, \bar v)\Bigr\rangle = (-1)^{k_2/2}
\frac{\zeta(2u) \G(s) \G(w)}{2^{k-2+2u} \pi^{k/2-1+2u}} L(f\otimes E(\cdot,v),w)
\end{equation}
for $k=k_1+k_2$ and $k_1, k_2 \gqs 0$. By comparing Euler products as in \cite[p 232]{Iwto}, for example,
\begin{equation}\label{lll}
L(f\otimes E(\cdot,v),w) = L(f,s) L(f,w)/\zeta(2u).
\end{equation}
Hence (\ref{ll}) and (\ref{lll}) complete the proof of the proposition.
\end{proof}

\noindent
{\bf Remark.} With (\ref{lower}), (\ref{raise})
in Lemma \ref{label} we obtain
$$
\Bigl\langle f, y^{-k/2}E^*_{k_1}(\cdot, \bar u)E^*_{k_2}(\cdot, \bar v)\Bigr\rangle = (-1)^{k_2/2}\frac{\G(u+|k_1|/2)}{\G(u+k_1/2)}\frac{\G(v+|k_2|/2)}{\G(v+k_2/2)} \Bigl\langle f, y^{-k/2}E^*_{k}(\cdot, \bar u)E^*_{0}(\cdot, \bar v)\Bigr\rangle
$$
for all $k_1, k_2 \in 2\Z$ with $k=k_1+k_2$ (removing the restriction $k_1, k_2-2 \gqs 0$).

\subsection{Proof of Theorem \ref{dsdw}} \label{proof}
\begin{proof}
 We may write $\doh_k(z,s)$ in terms of the basis $\mathcal B_k$:
\begin{eqnarray}
  \doh_k(z,s) &=& \sum_{f \in \mathcal B_k} \s{\doh_k(\cdot,s)}{f} \s{f}{f}^{-1} f(z) \nonumber\\
   &=& \sum_{f \in \mathcal B_k} L^*(f,s) \s{f}{f}^{-1} f(z). \label{entire}
\end{eqnarray}
Equation (\ref{entire}) makes it clear that $\doh_k(z,s)$ is an entire function of $s$.
Also with (\ref{entire}) we obtain
\begin{eqnarray}
  \s{\doh_k(\cdot,s)}{\doh_k(\cdot,\overline{w})} &=& \sum_{f,g \in \mathcal B_k} L^*(f,s) L^*(g,w) \s{f}{f}^{-1} \s{g}{g}^{-1} \s{f}{g} \nonumber\\
   &=& \sum_{f \in \mathcal B_k} L^*(f,s) L^*(f,w) \s{f}{f}^{-1}. \label{qwerty}
\end{eqnarray}
Since
$$
T_l \doh_k(z,s) = \sum_{f \in \mathcal B_k} \lambda_f(l) L^*(f,s) \s{f}{f}^{-1} f(z)
$$
we find
\begin{equation}\label{pro}
\s{T_l\doh_k(\cdot,s)}{\doh_k(\cdot,\overline{w})} = \sum_{f \in \mathcal B_k} \lambda_f(l) L^*(f,s) L^*(f,w) \s{f}{f}^{-1}.
\end{equation}
Use Proposition \ref{iprodprop} to express the product of $L$-functions in (\ref{pro}) as an inner product
 where $u, v$ are given by (\ref{uvsw})
and $k_1,k_2 \in 2\Z_{\gqs 0}$ satisfy $k_1+k_2=k$. We see that
\begin{eqnarray*}
 2 \cdot \pi^{k/2} \s{T_l\doh_k(\cdot,s)}{\doh_k(\cdot,\overline{w})} &=& (-1)^{k_2/2} \sum_{f \in \mathcal B_k} \lambda_f(l) \Bigl\langle f, y^{-k/2} E^*_{k_1}(\cdot, \bar u)E^*_{k_2}(\cdot, \bar v)\Bigr\rangle \s{f}{f}^{-1} \\
   &=& (-1)^{k_2/2} \Bigl\langle T_l \mathcal P  , y^{-k/2} E^*_{k_1}(\cdot, \bar u)E^*_{k_2}(\cdot, \bar v)\Bigr\rangle
\end{eqnarray*}
for
$$
\mathcal P :=  \sum_{f \in \mathcal B_k} \s{f}{f}^{-1} f.
$$
Then
$$
\s{f}{T_l \mathcal P} =\lambda_f(l) = \Bigl\langle f, P_l (4\pi l)^{k-1}/\G(k-1) \Bigr\rangle 
$$
 for all $f \in \mathcal B_k$  so we must have $T_l \mathcal P = P_l(4\pi l)^{k-1}/\G(k-1)$.  This finishes the proof of Theorem \ref{dsdw}.
\end{proof}

\section{A formula for the inner product $\bigl\langle T_l \doh_k(\cdot,s), \doh_k(\cdot,\overline w)\bigr\rangle$} \label{three}

\subsection{Eisenstein series at integer values of $s$}
For $k,h \in \Z$ and $u \in \Z_{\geqslant 0}$ define $h^*  :=  |h-1/2| -1/2 \ $ and
\begin{equation}
\label{defak}
\mathcal A^k_h(u)  :=  \frac{(-1)^{k/2}}{u!} \frac{\G(h-k/2+u)}{\G(h-k/2)}\frac{\G(h+|k|/2)}{\G(h+k/2-u)}.
\end{equation}
 It may be checked, working case by case, that
 \begin{equation}\label{vanishinga}
\mathcal A^k_h(u) \neq 0 \iff 0 \leqslant u \leqslant k/2-1- h^*  \quad \text{ for } \quad h^*  < k/2.
\end{equation}
Similarly,  when $h^*  \geqslant k/2$ we have $\mathcal A^k_h(u) \neq 0$ if and only if \  $0 \leqslant u \leqslant k/2+ h^* $.

\begin{theorem}
For all $k \in 2\Z$
and $h \in \Z$,
\begin{multline} \label{mnk}
   E^*_k(z,h)= \theta_k(h) y^h+\theta_k(1-h) y^{1-h} +
   \sum_{m \in \Z_{>0}} \frac{ \sigma_{2h-1}(|m|)}{|m|^h} \ e^{2\pi i m z} \sum_{u=0}^{h^*  +k/2}  \mathcal A^k_h(u) \cdot (4\pi |m|y)^{-u+k/2}\\
+
   \sum_{m \in \Z_{<0}} \frac{ \sigma_{2h-1}(|m|)}{|m|^h} \ e^{2\pi i m \overline{z}} \sum_{u=0}^{h^* -k/2}  \mathcal A^{-k}_h(u) \cdot (4\pi |m|y)^{-u-k/2}.
\end{multline}
\end{theorem}
\begin{proof}
Begin with the expansion (\ref{eisexp}). The Whittaker function may be expressed in terms of exponential functions at integer values  $s=h$. This yields (\ref{mnk}) for $k=0$. Applying the raising and lowering operators and induction on $k$ completes the proof. See \cite{Oeis} for more details.
\end{proof}

We shall be interested in the case when
there are no terms in (\ref{mnk}) with $m<0$. This happens exactly when $h^* -k/2 <0$.
Therefore, for $u, v \in \Z$, the product $E^*_{k_1}(z, u)E^*_{k_2}(z, v)$ appearing on the right side of (\ref{deqn}) will only have terms involving $e^{2\pi i n x}$ for $n \gqs 0$  if and only if
\begin{equation}
1-k_1/2 \lqs u \lqs k_1/2 \quad \text{and} \quad
1-k_2/2 \lqs v \lqs k_2/2.
\label{d7}
\end{equation}
Throughout the paper we shall use the correspondence $(u,v) \leftrightarrow (s,w)$ that we have already met in (\ref{uvsw}) with
$$
s=u-v+k/2, \quad w=u+v+k/2-1.
$$
Note the symmetries:
\begin{eqnarray*}
  s \to k-s & \iff & (u,v) \to (v,u) \\
  w \to k-w & \iff & (u,v) \to (1-v,1-u) \\
  u \to 1-u & \iff & (s,w) \to (k-w,k-s) \\
  v \to 1-v & \iff & (s,w) \to (w,s).
\end{eqnarray*}

\begin{lemma} \label{1swk}
For $u, v \in \Z$ and $k$ a positive even integer, there exist positive even $k_1$, $k_2$ satisfying $k_1+k_2=k$ and (\ref{d7}) if and only if
\begin{equation}
1 \lqs s, w  \lqs k-1 \quad \text{and} \quad s \not\equiv w \mod 2.
\label{sw7}
\end{equation}
\end{lemma}
\begin{proof}
Note that $u, v \in \Z$ exactly when $s,w$ are integers of opposite parity. If $u,v$ satisfy (\ref{d7}) then
\begin{equation}
2-k/2 \lqs u+v \lqs k/2 \quad \text{and} \quad
1-k/2 \lqs u-v \lqs k/2-1
\label{uvd}
\end{equation}
and (\ref{sw7}) follows.
Conversely, suppose (\ref{sw7}) holds. Then so does (\ref{uvd}) and consequently
\begin{equation*}
1-k/2 \lqs (u-1/2)+(v-1/2) \lqs k/2-1 \quad \text{and} \quad
1-k/2 \lqs (u-1/2)-(v-1/2) \lqs k/2-1
\end{equation*}
so that $|(u-1/2) \pm (v-1/2)| \lqs k/2-1$. Hence
$|u-1/2|+ |v-1/2| \lqs k/2-1$ and $u^* + v^* \lqs k/2-2$. Thus, there exist positive, even $k_1, k_2$
so that $u^* <k_1/2$, $v^* < k_2/2$ and $k_1+k_2=k$. This is equivalent to (\ref{d7}).
\end{proof}

\subsection{Holomorphic projection}
The holomorphic Eisenstein series is
$$
E_k(z):= \sum_{\g \in \G_\ci \backslash \G} \frac{1}{j(\g, z)^k} = \frac{1}{2} \sum_{c,d \in \Z \atop{(c,d)=1}} \frac{1}{(cz+d)^k},
$$
see for example \cite[p13]{Za},
converging for $4 \leqslant k \in 2\Z$ to a  modular form in the space $M_k(\G)$ of holomorphic, weight $k$ functions with possible polynomial growth at cusps. It has the Fourier expansion
\begin{equation}\label{ek}
E_k(z)=1-\frac{2k}{B_k} \sum_{m=1}^\infty \sigma_{k-1}(m) e^{2\pi i m z}.
\end{equation}
We recall a result of Sturm \cite{St}, extended by Zagier \cite[Appendix C]{Za}.
\begin{lemma}\label{holproj}
Suppose $F:\H \to \C$ is smooth, weight $k$, satisfies $F(z) \ll y^{-\epsilon}$ as $y\to \infty$ and has the expansion
$$
F(z)=\sum_{l \in \Z} F_l(y) e^{2\pi i l x}
$$
then
\begin{equation}\label{holp}
\s{F}{P_l} = \int_0^\infty F_l(y) e^{-2 \pi l y} y^{k-2} \, dy.
\end{equation}
\end{lemma}
The significance of Lemma \ref{holproj} and (\ref{holp}) are that they allow us to calculate the Fourier coefficients of $\pi_{hol}(F)$, the projection  of $F$ into the space $S_k$ with respect to the Petersson inner product. Thus
$$
\pi_{hol}(F) = \frac{1}{(k-2)!} \sum_{l=1}^\infty (4\pi l)^{k-1} \s{F}{P_l} e^{2\pi i l z} \quad \in S_k.
$$

\subsection{Proof of Theorem \ref{rmrn2}} \label{main}
Since
$$
E^*_{k_1}(z, u) = \theta_{k_1}(u) y^u + \theta_{k_1}(1-u) y^{1-u} + O(e^{-2\pi y})
$$
as $y \to \infty$ we have
\begin{eqnarray*}
  y^{-k/2}E^*_{k_1}(z,  u)E^*_{k_2}(z,  v) &=& \theta_{k_1}(u) \theta_{k_2}(v) y^{w+1-k} + \theta_{k_1}(u) \theta_{k_2}(1-v) y^{s+1-k} \\
  & & \quad + \theta_{k_1}(1-u) \theta_{k_2}(v) y^{1-s} + \theta_{k_1}(1-u) \theta_{k_2}(1-v) y^{1-w} + O(e^{-2\pi y})
\end{eqnarray*}
Thus, for $1 \lqs s,w \lqs k-1$, the function $F(z) := y^{-k/2}E^*_{k_1}(z,  u)E^*_{k_2}(z,  v)$ satisfies the conditions of Lemma \ref{holproj} except for the four cases when $s$ or $w$ equals $1$ or $k-1$. We may subtract a multiple of $E_k$ in these cases to remove the constant term. Recalling (\ref{poincare}) and noting that $\s{E_k}{P_l}=0$,
\begin{eqnarray*}
\bigl\langle P_l , y^{-k/2}E^*_{k_1}(z,  u)E^*_{k_2}(z,  v) \bigr\rangle
& = & \bigl\langle P_l , y^{-k/2}E^*_{k_1}(z,  u)E^*_{k_2}(z,  v)
- \lambda(s,w) E_k \bigr\rangle\\
& = & \int_0^\infty F_l(y) e^{-2 \pi l y} y^{k-2} \, dy
\end{eqnarray*}
for
$$
\lambda(s,w):=\delta_{w,k-1} \theta_{k_1}(u) \theta_{k_2}(v)+
\delta_{s,k-1} \theta_{k_1}(u) \theta_{k_2}(1-v)+
\delta_{s,1} \theta_{k_1}(1-u) \theta_{k_2}(v)+
\delta_{w,1} \theta_{k_1}(1-u) \theta_{k_2}(1-v)
$$
and
$$
y^{-k/2}E^*_{k_1}(z,  u)E^*_{k_2}(z,  v)
- \lambda(s,w) E_k = \sum_{l \in \Z} F_l(y) e^{2\pi i l x}.
$$
With the expansion (\ref{mnk}),
$$
E_{k}^*(z, u)=\sum_{n=0}^{\infty} e_{k}(n; y, u)e^{2 \pi i n x}
$$
when $1-k/2 \lqs u \lqs k/2$
for
\begin{eqnarray*}
e_k(0; y, u) & = & \theta_k(u)y^u+\theta_k(1-u)y^{1-u}, \\
e_k(n; y, u) & = & \frac{\sigma_{2u-1}(n)}{n^u} e^{-2 \pi n y} \sum_{r=0}^{k/2-1-u^*} \mathcal A_u^k(r)(4
\pi n y)^{-r+k/2} \qquad (n>0).
\end{eqnarray*}
Thus $F_l(y)$ breaks up into three natural pieces $\Lambda_1(y)+\Lambda_2(y)+\Lambda_3(y)$ with
\begin{eqnarray*}
\Lambda_1(y) & = & y^{-k/2}e_{k_1}(0; y, u)e_{k_2}(l; y, v) + y^{-k/2}e_{k_1}(l; y, u)e_{k_2}(0; y, v), \label{s1}\\
\Lambda_2(y) & = & \sum_{n=1}^{l-1} y^{-k/2}e_{k_1}(n; y, u)e_{k_2}(l-n; y, v),\label{s2}\\
\Lambda_3(y) & = & - \lambda(s,w)\frac{(2\pi i)^{k}}{\G(k) \zeta(k)}  \sigma_{k-1}(l) e^{-2\pi ly}.\label{s3}
\end{eqnarray*}
Setting
\begin{equation*}\label{omega1}
\Psi_i(s,w;l):=(-1)^{k_2/2} 2^{k-1} \pi^{k/2-1} l^{k-1} \int_0^\infty \Lambda_i(y) e^{-2 \pi l y} y^{k-2} \, dy
\end{equation*}
we have by Theorem \ref{dsdw} that
\begin{equation}\label{omega2}
(k-2)! 2^{2-k}\Bigl\langle T_l \doh_k(\cdot,s), \ \doh_k(\cdot, w)\Bigr\rangle  = \Psi := \Psi_1 + \Psi_2 + \Psi_3.
\end{equation}

\vskip 3mm
With Propositions \ref{psi111}, \ref{psi333} and \ref{psi222} below we compute the right side of (\ref{omega2}) and complete the proof of Theorem \ref{rmrn2}.

\vskip 3mm
\noindent
{\bf Remark.}
The $n$th Rankin-Cohen bracket $[f,g]_n$ of $f \in M_{k_1}$, $g \in M_{k_2}$ is (see for example \cite[p. 249]{Za})
\begin{equation}\label{rcb}
    [f,g]_n := \sum_{n_1+n_2=n} (-1)^{n_1} \binom{n+k_1-1}{n_1} \binom{n+k_2-1}{n_2} f^{(n_1)} g^{(n_2)}
\end{equation}
and we have $[f,g]_n \in M_{k_1+k_2+2n}$. In \cite{Za3} Zagier proves the identity
\begin{equation}\label{zaggg}
    \s{f}{[E_{k_1}, E_{k_2}]_n} = (-1)^{k_1/2}(2\pi i)^n 2^{3-k} \frac{k_1 k_2}{B_{k_1} B_{k_2}}\frac{\G(k-1)}{\G(k-n-1)} L^*(f,n+1) L^*(f,n+k_2)
\end{equation}
where $k=k_1+k_2+2n$ and $f \in \mathcal B_k$. (The $n=0$ case is due to Rankin.)
Replace $k_1,k_2,n$ in (\ref{zaggg}) by $2u,2v,k-1-w$ respectively and
compare the result with (\ref{ziggg}) to see that
\begin{equation}\label{pih}
\pi_{hol}\left(y^{-k/2}E^*_{k_1}(z,  u)E^*_{k_2}(z,  v)\right) = \frac{(-1)^{k_2/2+u} 2^{k-4}\pi^{k/2} \G(w) B_{2u} B_{2v}}{ (2\pi i)^{k-1-w} \G(k-1) uv}  [E_{2u}, E_{2v}]_{k-1-w}
\end{equation}
for $u,v \gqs 2$, $u+v<k/2$. In showing (\ref{pih}) we also utilized the functional equation
\begin{equation}\label{feqqq}
L^*(f,k-s)=(-1)^{k/2} L^*(f,s)
\end{equation}
as in, for example \cite[(46)]{Za2} (we give a novel new proof of (\ref{feqqq}) in section \ref{funeq9}).
Kohnen and Zagier use (\ref{zaggg}) to prove Theorems \ref{rmrn} and  \ref{per} below, see \cite[p 214]{KZ}.

\subsubsection{Calculating $\Psi_1(s,w;l)$}
\begin{prop} \label{psi111}
For $s,w$ of opposite parity and satisfying $1 \lqs s,w \lqs k-1$
\begin{multline*}
\Psi_1(s,w;l) =
    \sigma_{2v-1}(l)
  \Bigl[\rho(2u)l^{k-1-w}  \G(s)\G(w)
     +
     \rho(2-2u)l^{s-1} \G(k-s)\G(k-w)\Bigr] \\
   +  (-1)^{k/2} \sigma_{2u-1}(l)
  \Bigl[ \rho(2v)l^{k-1-w}  \G(k-s)\G(w)
     +
     \rho(2-2v)l^{k-1-s} \G(s)\G(k-w)\Bigr].
\end{multline*}
\end{prop}
\begin{proof}
Write
\begin{eqnarray*}
  F_{k_1,k_2}(l;u,v) &:=& \int_0^\infty \theta_{k_1}(u) y^u \left( \frac{\sigma_{2v-1}(l)}{l^v} e^{-2 \pi l y} \sum_{r=0}^{k_2/2-1-v^* } \mathcal A_v^{k_2}(r)(4
\pi l y)^{-r+k_2/2}\right) e^{-2 \pi l y} y^{k/2-2} \, dy \\
   &=& \theta_{k_1}(u) \frac{\sigma_{2v-1}(l)}{l^v} \sum_{r=0}^{k_2/2-1-v^* } \frac{\mathcal A_v^{k_2}(r)}{(4\pi l)^{u+k/2-1}} \int_0^\infty  (4\pi l y)^{u+k/2-1}  (4
\pi l y)^{-r+k_2/2} e^{-4 \pi l y}  \, \frac{dy}y \\
   &=& \frac{ (4\pi)^v \theta_{k_1}(u) \sigma_{2v-1}(l)}{(4\pi l)^{k/2-1+u+v}} \sum_{r=0}^{k_2/2-1-v^* } \mathcal A_v^{k_2}(r) \G(k/2+k_2/2-1+u-r).
\end{eqnarray*}
We have
\begin{equation}\label{akvr}
\mathcal A_v^{k_2}(r) = (-1)^{k_2/2+r} r! \binom{k_2/2-v}{r} \binom{k_2/2-1+v}{r}
\end{equation}
(by (\ref{vanishinga}) it is nonzero exactly for $0 \lqs r \lqs k_2/2-1-v^*$), so that
\begin{multline*}
  \sum_{r=0}^{k_2/2-1-v^* } \mathcal A_v^{k_2}(r) \G(k/2+k_2/2-1+u-r) \\
  = (-1)^{k_2/2}\sum_{r=0}^{k_2/2-1-v^* } (-1)^{r}  \binom{k_2/2-v}{r} \binom{k_2/2-1+v}{r} r! (k/2+k_2/2-2+u-r)! \\
   = (-1)^{k_2/2} (v+k_2/2-1)! (u-v+k/2-1)! \sum_{r=0}^{k_2/2-1- v^*} (-1)^{r}  \binom{k_2/2-v}{k_2/2-v-r} \binom{k/2+k_2/2-2+u-r}{u-v+k/2-1} \\
   = (-1)^{v} (v+k_2/2-1)! (u-v+k/2-1)! \sum_{t} (-1)^{t}  \binom{k_2/2-v}{t} \binom{(k/2-2+u+v)+t}{u-v+k/2-1}.
\end{multline*}
Using the identity (which may be proved  as in Lemma \ref{zl})
\begin{equation}\label{combin}
\sum_t (-1)^t \binom{a}{t} \binom{b+t}{c} = (-1)^a \binom{b}{c-a}
\end{equation}
and $\zeta(2n)= 2^{2n-1}\pi^{2n} \rho(2n)$ we obtain
\begin{equation}
F_{k_1, k_2}(l;u, v)  = \frac{(-1)^{k_2/2}  \rho(2u) \sigma_{2v-1}(l)}
{2 (4\pi l)^{k/2-1} l^{u+v}} \G(s) \G(w).
\label{d9}
\end{equation}
Clearly
$$
\int_0^\infty \Lambda_1(y) e^{-2 \pi l y} y^{k-2} \, dy = F_{k_1, k_2}(l;u, v) + F_{k_1, k_2}(l;1-u, v)+
F_{k_2, k_1}(l;v,u)+F_{k_2, k_1}(l;1-v,u)
$$
and the Proposition follows.
\end{proof}

\subsubsection{Calculating $\Psi_3(s,w;l)$}
\begin{prop} \label{psi333}
For $s,w$ of opposite parity and satisfying $1 \lqs s,w \lqs k-1$
\begin{multline*}
\Psi_3(s,w;l) = - \frac{\sigma_{k-1}(l)}{(k-1) \rho(k)} \Bigl[ \left(\delta_{w,1}(-1)^{(k-s)/2} + \delta_{w,k-1}(-1)^{s/2}\right)\G(s)\G(k-s)\rho(s)\rho(k-s) \\
+ \left(\delta_{s,1}(-1)^{(k-w)/2} + \delta_{s,k-1}(-1)^{w/2}\right)\G(w)\G(k-w)\rho(w)\rho(k-w)\Bigr].
\end{multline*}
\end{prop}
\begin{proof}
It is easy to show that
$$
\int_0^\infty \Lambda_3(y) e^{-2 \pi l y} y^{k-2} \, dy = -\lambda(s,w) \frac{(-1)^{k/2} (2\pi)^k \sigma_{k-1}(l)}{(k-1) \zeta(k) (4\pi l)^{k-1}}.
$$
Also $\lambda(s,w)$ simplifies a good deal. For example, when $w=k-1$ we have $u=s/2$ and $v=(k-s)/2$. Since $k_1$ is chosen (recall Lemma \ref{1swk}) so that $u^*=s/2-1<k_1/2$ it follows that $k_1=s$ and  $k_2=k-s$. Therefore $$
w=k-1 \implies \theta_{k_1}(u)\theta_{k_2}(v)= \pi^{-k/2} \G(s)\G(k-s)\zeta(s)\zeta(k-s).
$$
The other terms in $\lambda(s,w)$ behave similarly and
$$
\lambda(s,w)=\pi^{-k/2}\left[ (\delta_{w,1} + \delta_{w,k-1})\G(s)\G(k-s)\zeta(s)\zeta(k-s) + (\delta_{s,1} + \delta_{s,k-1})\G(w)\G(k-w)\zeta(w)\zeta(k-w)\right].
$$
Finally, noting that for any $n \in 2\Z$
$$
\frac{\zeta(n) \zeta(k-n)}{\zeta(k)} = \frac{\rho(n) \rho(k-n)}{2\rho(k)}
$$
we obtain the proposition.
\end{proof}


\subsubsection{Calculating $\Psi_2(s,w;l)$}
Recall the definition (\ref{zdef}) of the polynomial  $Z_{s,w}(x)$.
\begin{prop} \label{psi222}
For $s,w$ of opposite parity and satisfying $1 \lqs s,w \lqs k-1$
\begin{equation}
\Psi_2(s,w;l) = 2 (-1)^{k/2} (k-2)! l^{k-1-w} \sum_{n=1}^{l-1} \sigma_{2u-1}(n)\sigma_{2v-1}(l-n)  Z_{s,w}(n/l).
\end{equation}
\end{prop}
\begin{proof}
For each $n$ between $1$ and $l-1$
\begin{multline*}
  \int_0^\infty e_{k_1}(n; y, u)e_{k_2}(l-n; y, v) e^{-2 \pi l y} y^{k/2-2} \, dy
  = \frac{\sigma_{2u-1}(n)}{n^u} \frac{\sigma_{2v-1}(l-n)}{(l-n)^v} \\
 \times \sum_{a=0}^{k_1/2-1- u^*}\sum_{b=0}^{k_2/2-1- v^*} \mathcal A_u^{k_1}(a) \mathcal A_v^{k_2}(b)
 \int_0^\infty  e^{-4 \pi l y} y^{k/2-1} (4\pi n y)^{-a+k_1/2} (4\pi (l-n) y)^{-b+k_2/2}    \, \frac{dy}y  \\
   = \frac{\sigma_{2u-1}(n)\sigma_{2v-1}(l-n) l^{-w}}{(4\pi)^{k/2-1} }  \sum_{a,b} \mathcal A_u^{k_1}(a) \mathcal A_v^{k_2}(b) \left(\frac{n}{l} \right)^{k_1/2-u-a} \left(1-\frac{n}{l} \right)^{k_2/2-v-b} \G(k-1-a-b).
\end{multline*}
Thus we have
$$
\Psi_2(s,w;l) = 2 \cdot  l^{k-1-w} \sum_{n=1}^{l-1} \sigma_{2u-1}(n)\sigma_{2v-1}(l-n) Y_{s,w}(n/l)
$$
on setting
\begin{equation}\label{ysw}
Y_{s,w}(x):= (-1)^{k_2/2} \sum_{a,b} \mathcal A_u^{k_1}(a) \mathcal A_v^{k_2}(b) x^{k_1/2-u-a} (1-x)^{k_2/2-v-b} (k-2-a-b)!
\end{equation}
To complete the proof we need to demonstrate that
$
 Y_{s,w}(x) = (-1)^{k/2}(k-2)! Z_{s,w}(x)
$.
We see from (\ref{omega2}) and Propositions \ref{psi111}, \ref{psi333} that $\Psi_2(s,w;l)$ must be independent of the choice of $k_1$, $k_2$ satisfying (\ref{d7}).
Choose $k_1$ so that $k_1/2-1- u^*=0$ to simplify (\ref{ysw}). Thus $a=0$ and, with (\ref{akvr}), $\mathcal A_u^{k_1}(0)=(-1)^{k_1/2}$. Hence
$$
Y_{s,w}(x)= (-1)^{k/2}x^{k_1/2-u} \sum_{b=0}^{k_2/2-1- v^*}  \mathcal A_v^{k_2}(b) (1-x)^{k_2/2-v-b} (k-2-b)!
$$
Use the  formula (\ref{akvr}) and the binomial expansion of $(1-x)^{k_2/2-v-b}$  to obtain
\begin{multline*}
Y_{s,w}(x) = (-1)^{k/2+k_2/2} (k_2/2-v)! x^{k_1/2-u}
\sum_{r=0}^{k_2/2-v-b} (k-k_2/2-2+v+r)! (-x)^r/r! \\
\times \sum_b (-1)^b \binom{k_2/2-1+v}{k_2/2-1+v-b} \binom{k-2-b}{k-k_2/2-2+v+r}.
\end{multline*}
The combinatorial identity  (\ref{combin}), substituting $t$ for $k_2/2-1+v-b$, evaluates the inner sum over $b$ and
$$
Y_{s,w}(x) = (-1)^{u^*+1} \frac{(k-2)!}{ \binom{k-2}{k/2+u^*-v}} x^{u^*+1-u}
\sum_r (-x)^r \binom{k/2-1-u^*-v}{r} \binom{k/2-1+u^*+v+r}{k/2-2-u^*+v}
$$
where we replaced $k_1/2$ by $u^*+1$ and $k_2/2$ by $k/2-u^*-1$. Recall that if $u \lqs 0$ (equivalent to $s+w<k$) then $u^*=-u$. If $u \gqs 1$ ( $s+w>k$) then $u^*=u-1$. Therefore $(-1)^{k/2} Y_{s,w}(x)/(k-2)!$ equals
\begin{eqnarray}
    (-1)^{(s+w-1)/2} x^{k-s-w} \binom{k-2}{w-1}^{-1}   \sum_t (-x)^t \binom{s-1}{t} \binom{k-1-s+t}{w-1}  & & \text{if \ \ } s+w<k, \label{zz1}\\
    (-1)^{(s+w+1)/2}  \binom{k-2}{s-1}^{-1}  \sum_r (-x)^r \binom{k-1-w}{r} \binom{w-1+r}{k-1-s}  & & \text{if \ \ } s+w>k. \label{zz2}
\end{eqnarray}
The change of variables $r=t+k-s-w$ in (\ref{zz1}) implies (\ref{zz1}) equals (\ref{zz2}) for all $s$ and $w$.
This completes the proof of the proposition.
\end{proof}

\section{Applications of Theorem \ref{rmrn2}}
\subsection{The Kohnen Zagier formula}
Specializing Theorem \ref{rmrn2} to $l=1$ and with $s$, $w$ replaced by $m+1$ and $n+1$ we retrieve Kohnen and Zagier's formula.
To state their result, recall that $\tilde m:=k-2-m$, $\tilde n:=k-2-n$.
\begin{theorem} \cite{KZ} \label{rmrn}
For integers $m,n$ of opposite parity with $0 \lqs m,n \lqs k-2$
\begin{eqnarray}
2^{2-k}(k-2)!\bigl \langle R_m, R_n \bigr \rangle & = &
 \rho(m - \tilde n + 1) m! n!
 +
 \rho(-m +  \tilde n + 1)  \tilde m! \tilde n! \nonumber
\\ &  & +
(-1)^{k/2} \rho(m -  n + 1) m! \tilde n!
  +
(-1)^{k/2} \rho(-m + n + 1)  \tilde m! n! \label{www}
\end{eqnarray}
where, if $m$ or $n$ equals $0$ or $k-2$, we must add
$$
(-1)^{\frac{m-1}{2}}\frac{m! \tilde m! \rho(m+1) \rho(\tilde m +1)}{(k-1) \rho(k)}
\bigl( (-1)^{k/2}\delta_{n,0} +  \delta_{n,k-2} \bigr) +
(-1)^{\frac{n-1}{2}}\frac{n! \tilde n! \rho(n+1) \rho(\tilde n +1)}{(k-1) \rho(k)}
\bigl( (-1)^{k/2}\delta_{m,0} +  \delta_{m,k-2} \bigr)
$$
to the right side of (\ref{www}).
\end{theorem}

 Lanphier in  \cite{Lanp} uncovers combinatorial connections between the raising operators (\ref{wlr}) of Maass and Shimura and the Rankin-Cohen bracket (\ref{rcb}). This leads to another proof of Theorem \ref{rmrn}. For example, with (\ref{qwerty}) we may also express Theorem \ref{rmrn2}, specialized to $l=1$, as
$$
\sum_{f \in \mathcal B_k} \frac{L^*(f,s) L^*(f,w)}{\s{f}{f}} = \frac{\Psi(s,w;1)}{2^{2-k} (k-2)!}.
$$
This is \cite[Corollary 3]{Lanp}\footnote{The right side in the statement of that Corollary should be multiplied by a missing $-B_l/(2l)$.}.
The general case of Theorem \ref{rmrn2} may be formulated as
\begin{equation}\label{sm}
\sum_{f \in \mathcal B_k} \frac{\lambda_f(l) L^*(f,s) L^*(f,w)}{\s{f}{f}} = \frac{\Psi(s,w;l)}{2^{2-k} (k-2)!}
\end{equation}
with the rational number $\Psi(s,w;l)$  given by the right side of (\ref{bigformula}). See \cite{Za4} for a different treatment of the left side of (\ref{sm}).

\subsection{Ramanujan-style identities} \label{arit}
The right side of Theorem \ref{rmrn2} must evaluate to $0$ for $k=4,6,8,10,14$ since, for these weights, cusp forms do not exist. The resulting identities may be verified and serve to check the statement of the theorem.  For $k=12$ we must have $\bigl\langle T_l \doh(\cdot,s), \doh(\cdot, w)\bigr\rangle/\bigl\langle  \doh(\cdot,s), \doh(\cdot, w)\bigr\rangle = \tau(l)$, Ramanujan's tau function. Choosing $u=v=1$ for example yields
\begin{equation}\label{raman}
(6l-1) \sigma_1(l) -5\sigma_3(l)+  12\sum_{n=1}^{l-1} \sigma_1(n) \sigma_1(l-n)  = 0
\end{equation}
when $k=4$, an identity of Ramanujan \cite[(2)]{Ra}.  Another proof of (\ref{raman}) using holomorphic projection appears in \cite[p 288]{Za}.
For $6 \lqs k \lqs 14$ we obtain
 equalities involving only  $\sigma_1$ and $\tau$:
\begin{eqnarray*}
 \sum_{n=1}^{l-1} \sigma_1(n) \sigma_1(l-n) \Bigl[ l-2n\Bigr] &=& 0, \quad (k=6) \\
(l-1)l^2 \sigma_1(l) +12 \sum_{n=1}^{l-1} \sigma_1(n) \sigma_1(l-n) \Bigl[ l^2-5ln +5n^2 \Bigr] &=& 0, \quad (k=8) \\
\sum_{n=1}^{l-1} \sigma_1(n) \sigma_1(l-n) \Bigl[ l^3-9l^2 n+21 l n^2-14 n^3 \Bigr] &=& 0, \quad (k=10) \\
   \frac 13 (5-2l)l^4 \sigma_1(l) -  20\sum_{n=1}^{l-1} \sigma_1(n) \sigma_1(l-n) \Bigl[l^4-14 l^3 n +56 l^2 n^2 -84 l n^3 +42 n^4\Bigr]
   &=& \tau(l), \quad (k=12)\\
\sum_{n=1}^{l-1} \sigma_1(n) \sigma_1(l-n) \Bigl[ l^5-20 l^4 n+120 l^3 n^2 -300 l^2 n^3 +330 l n^4-132 n^5 \Bigr] &=& 0. \quad (k=14)
\end{eqnarray*}
Niebur's  formula \cite{Ni}
$$
 l^4 \sigma_1(l) -  24\sum_{n=1}^{l-1} \sigma_1(n) \sigma_1(l-n) \Bigl[18 l^2 n^2 -52 l n^3 +35 n^4\Bigr]
   = \tau(l)
$$
is a linear combination of the above equalities with $k=6,8,10,12$. See also \cite[(9.5c)]{OS}, for example.

\subsection{The Periods Theorem} \label{manin}
Let $f\in \mathcal B_k$ and $K_f$ the field obtained by adjoining the coefficients $a_f(n)$ to $\Q$. Then $K_f \subset \R$   because $T_n$ is self adjoint. (From (\ref{saj}) below it follows that $K_f$ is totally real.) Let $g_j$ for $1 \lqs j \lqs d$ be a Miller basis for $S_k$, see \cite[Chapter X, Theorem 4.4]{La}. The Fourier coefficients of $g_i$ are in $\Z$ and of the first $d$  coefficients, only the $j$th is non-zero (it equals $1$).

Since $T_n$ maps the column vector $(g_1, \cdots, g_d)^T$ to $[T_n](g_1, \cdots, g_d)^T$ for $[T_n]$ a $d\times d$ matrix with entries in $\Z$, we see that the eigenvalues of $T_n$ are roots of a degree $d$ polynomial and hence the degree of any element of $K_f$ over $\Q$ is at most $d$. This is \cite[Theorem 3]{Za}.
Also
\begin{equation}\label{miller}
f=\sum_{j=1}^d \lambda_f(j) g_j
\end{equation}
and it follows that $K_f$ is a finite extension of $\Q$ with $[K_f:\Q] \lqs d^d$.

\vskip 3mm
We now prove Manin's Periods Theorem \cite{Ma1}, in the slightly more precise form of \cite[p 202]{KZ}. See also Shimura's general result \cite[Theorem 1]{Sh}.

\begin{theorem}  \label{per}
Given $f \in \mathcal B_k$ there exist $\omega_+(f), \ \omega_-(f) \in \R$ such that $\omega_+(f) \omega_-(f) = \bigl\langle f, f \bigr\rangle$ and
$$
L^*(f,s)/\omega_+(f), \quad L^*(f,w)/\omega_-(f) \in K_f
$$
for all $s,w$ with $1 \lqs s, w \lqs k-1$ and $s$ even, $w$ odd.
\end{theorem}
\begin{proof}
Set
$$
H_{s,w}(z):=\pi_{hol}\Bigl[ (-1)^{k_2/2} y^{-k/2}E^*_{k_1}(z,  u)E^*_{k_2}(z,  v)/(2\pi^{k/2}) \Bigr] \in S_k(\G).
$$
Then, recalling Proposition \ref{iprodprop}, we have
$
L^*(f,s) L^*(f,w) = \bigl\langle f, H_{\bar s,\bar w}\bigr\rangle.
$
By Theorem \ref{dsdw} we know
$$
\Bigl\langle (-1)^{k_2/2} y^{-k/2}E^*_{k_1}(z,  u)E^*_{k_2}(z,  v)/(2\pi^{k/2}), P_l\Bigr\rangle
=\frac{(k-2)!}{(4\pi l)^{k-1}} \overline{\bigl\langle T_l \doh_k(\cdot,s), \doh_k(\cdot,\overline w)\bigr\rangle}
$$
which implies
$$
H_{s,w}(z) = \sum_{l=1}^\infty \overline{\bigl\langle T_l \doh_k(\cdot,s), \doh_k(\cdot,\overline w)\bigr\rangle} e^{2\pi i l z}.
$$
Hence, for $s,w$ of opposite parity satisfying $1\lqs s,w \lqs k-1$, Theorem \ref{rmrn2} shows that
$
H_{s,w}(z)$ has rational Fourier coefficients.
For any $g \in S_k$ with rational Fourier coefficients and $f \in \mathcal B_k$ we have
$
\bigl\langle f, g \bigr\rangle = c \bigl\langle f, f \bigr\rangle
$
with $c \in K_f$,  see Lemma \ref{shim} below. Thus
$$
L^*(f,k-2) L^*(f,k-1) = \bigl\langle f, H_{ k-2, k-1}\bigr\rangle = c_{f} \bigl\langle f, f \bigr\rangle
$$
for $c_f \in K_f$ and the left side is nonzero because the Euler products converge for $\Re(s)>k/2+1/2$.
Set
$$
\omega_+(f) := \frac{c_f \bigl\langle f, f \bigr\rangle}{L^*(f,k-1)}, \quad \omega_-(f) := \frac{\bigl\langle f, f \bigr\rangle}{L^*(f,k-2)}.
$$
Then, for $s$ even and $1 < s < k-1$,
$$
\frac{L^*(f,s)}{\omega_+(f)} = \frac{L^*(f,s)L^*(f,k-1)}{c_f \bigl\langle f, f \bigr\rangle} = \frac{\bigl\langle f, H_{s, k-1}\bigr\rangle}{c_{f} \bigl\langle f, f \bigr\rangle} = \frac{c'_{f} \bigl\langle f, f \bigr\rangle}{c_{f} \bigl\langle f, f \bigr\rangle}\in  K_f
$$
and similarly for $w$ odd, as required.
\end{proof}

The following lemma is implicit in the proofs of \cite{KZ}, \cite{Za3} and a special case of \cite[Lemma 4]{Sh}. Since the proof is short and instructive we include it for completeness.

\begin{lemma} \label{shim}
For any $g \in S_k$ with rational Fourier coefficients and $f \in \mathcal B_k$, a normalized Hecke form,
$$
\bigl\langle g, f \bigr\rangle / \bigl\langle f, f \bigr\rangle  \in K_f.
$$
\end{lemma}
\begin{proof}
Let $\sigma$ be any automorphism of $\C$, $z \mapsto z^\sigma$, necessarily fixing $\Q$. For any $h=\sum_{n=1}^\infty a(n) e^{2\pi i n z} \in S_k(\G)$ define $h^\sigma = \sum_{n=1}^\infty a(n)^\sigma e^{2\pi i n z}$.   Let $f \in \mathcal B_k$ and writing $f$ in terms of the Miller basis, as in (\ref{miller}), we find
\begin{equation}\label{saj}
  T_n(f^\sigma) = \sum_j \lambda_f(j)^\sigma \cdot T_n g_j = \sum_j \bigl( \lambda_f(j) \cdot T_n g_j \bigr)^\sigma = (T_n f)^\sigma = (\lambda_f(n) f)^\sigma =  \lambda_f(n)^\sigma f^\sigma.
\end{equation}
It follows that $f^\sigma \in \mathcal B_k$ also and thus $\sigma$ permutes the set $\mathcal B_k=\{f_i\}_{1 \lqs i \lqs d}$.
Let  $f=f_1$, say. By (\ref{miller}) we know $(f_1, \dots, f_d)^T=M (g_1, \dots, g_d)^T$ where the  $d \times d$ matrix $M$ has entries in $K=K_{f_1}K_{f_2} \cdots K_{f_d}$, as does $M^{-1}$. It follows that $g=\sum_i c_i f_i$ with  $c_i \in K$. Then $\bigl\langle g, f \bigr\rangle = c_1 \bigl\langle f, f \bigr\rangle$. Also, since $g^\sigma = g$,
$$
\bigl\langle g, f^\sigma \bigr\rangle = c_1^\sigma \bigl\langle f^\sigma, f^\sigma \bigr\rangle.
$$
Therefore $c_1^\sigma = c_1$ if $\sigma$ fixes the elements of $K_f$. Now $K$ is finite  extension of $K_f$, and normal since any embedding of $K$ in $\C$ permutes $\mathcal B_k$. Hence $c_1^\sigma = c_1$ for all $\sigma \in \operatorname{Gal}(K/K_f)$. The Galois correspondence then implies $c_1 \in K_f$.
\end{proof}

\subsection{Functional equations}
In this section we explore in detail the functional equations of both sides of (\ref{omega2}).
Define the symmetries $\alpha$, $\beta$ acting on pairs $(s,w) \in \C^2$ as follows:
$$
(s,w) \stackrel{\alpha}{\longrightarrow} (w,s), \qquad (s,w) \stackrel{\beta}{\longrightarrow} (k-s,w).
$$
They generate the Dihedral group with $8$ elements
$$
D_8 = \Bigl\langle \alpha, \beta : \alpha^2 = \beta^2 = (\alpha \beta)^4 = I \Bigr \rangle.
$$
Note that the effects of $\alpha$, $\beta$ on the pairs $(u,v)$ related to $(s,w)$ by (\ref{uvsw})  are
$$
(u,v) \longrightarrow (u,1-v), \qquad (u,v) \longrightarrow  (v,u)
$$
respectively. With  $\alpha$, $\beta$ acting on functions of $s,w$ via the left regular representation, we may describe the functional equations of
$\bigl\langle T_l \doh_k(\cdot,s), \doh_k(\cdot, w)\bigr\rangle$. With (\ref{entire}) and the functional equation (\ref{feqqq})
we obtain
\begin{equation}\label{feqqq9}
\doh_k(z,k-s)=(-1)^{k/2} \doh_k(z,s).
\end{equation}
Also we know that $T_l$ is self adjoint. Therefore
$$
\alpha \Bigl[\bigl\langle T_l \doh_k(\cdot,s), \doh_k(\cdot, w)\bigr\rangle\Bigr] = \overline{\bigl\langle T_l \doh_k(\cdot,s), \doh_k(\cdot, w)\bigr\rangle}, \quad \beta \Bigl[\bigl\langle T_l \doh_k(\cdot,s), \doh_k(\cdot, w)\bigr\rangle\Bigr] =
(-1)^{k/2} \bigl\langle T_l \doh_k(\cdot,s), \doh_k(\cdot, w)\bigr\rangle.
$$
Now for integers $s,w$ of opposite parity satisfying $1\lqs s,w \lqs k-1$ we see quickly from Propositions \ref{psi111} and \ref{psi333} that the $\Psi_i$ have the same functional equations for $i=1,3$:
\begin{equation}\label{oi}
\alpha \bigl[ \Psi_i(s,w;l)\bigr] = \Psi_i(s,w;l), \qquad \beta \bigl[ \Psi_i(s,w;l) \bigr] = (-1)^{k/2}\Psi_i(s,w;l).
\end{equation}
Hence (\ref{oi}) must be true for $i=2$ also. We verify this elegant symmetry directly.

\vskip 3mm
First, we note from the equality of (\ref{zz1}), (\ref{zz2}) that
$$
Z_{k-w,k-s}(x)=x^{s+w-k} Z_{s,w}(x).
$$
It follows that
$$
\beta  \alpha  \beta \bigl[ \Psi_2(s,w;l)\bigr] = \Psi_2(s,w;l).
$$
\begin{prop} \label{zl} We have
$$
Z_{k-s,w}(1-x)= (-1)^{k/2} Z_{s,w}(x).
$$
\end{prop}
\begin{proof}
With (\ref{zdef}) we need to verify
\begin{equation}\label{ver}
    (-1)^{s}    \sum_r (-1+x)^r \binom{k-1-w}{r} \binom{w-1+r}{s-1}
    =   \sum_r (-x)^r \binom{k-1-w}{r} \binom{w-1+r}{k-1-s}.
\end{equation}
We define a generating function $p(x,y)$ as follows
\begin{eqnarray}
  p(x,y) &:=& [1-x(1+y)]^{k-1-w}(1+y)^{w-1} \nonumber \\
   &=& \sum_r (-x(1+y))^r (1+y)^{w-1} \binom{k-1-w}{r} \nonumber\\
   &=& \sum_{r,t} (-x)^r  \binom{k-1-w}{r} \binom{w-1+r}{t} y^t. \label{qo1}
\end{eqnarray}
With the  identity
$
y^{k-2}p(1-x,1/y) = (-1)^{w+1} p(x,y)
$
we also find
\begin{equation}\label{qo2}
p(x,y) = (-1)^{w+1} \sum_{r,t} (-1+x)^r  \binom{k-1-w}{r} \binom{w-1+r}{t} y^{k-2-t}.
\end{equation}
Equating powers of $y$ in (\ref{qo1}), (\ref{qo2}) shows (\ref{ver}) and finishes the proposition's proof.
\end{proof}

As a consequence of Proposition \ref{zl} and (\ref{sigma}),
$$
\beta \bigl[ \Psi_2(s,w;l)\bigr] = (-1)^{k/2} \Psi_2(s,w;l).
$$
Since $\beta$, $\beta  \alpha  \beta$ generate $D_8$ we have verified (\ref{oi}) when $i=2$.

\section{Cohen's series representation} \label{cgt}

\subsection{}
In this section let $\G \subseteq \PSL_2(\R)$ be a Fuchsian group of the first kind, such as $\G_0(N)$, with fixed representatives for inequivalent cusps $\{\ca, \cb, \cc, \dots \}$. We  restrict our attention to the case where $\G$ has at least one cusp; the compact case will be similar.
The subgroup $\G_\ca$ of all
elements in $\G$ that fix $\ca$ is isomorphic to $\Z$.   There exists a scaling matrix $\sa
\in \SL_2(\R)$ so that $\sa \infty= \ca$ and
\begin{equation*}
  \sa^{-1} \G_\ca \sa =  \left\{ \left. \pm \begin{pmatrix} 1 & m \\ 0 & 1
\end{pmatrix}
\; \right| \; \ m\in {\Z}\right\}.
\end{equation*}
See \cite{Iwsp} for further details. Define the series
\begin{equation}\label{omega}
\Omega_\ca(z,\tau;s,k):= \sum_{\g \in \G} \frac 1 {(\sa^{-1}\g z-\overline{\tau})^{s} j(\sa^{-1}\g,z)^k}
\end{equation}
for $z \in \H$ and $ \tau \in \H \cup \R$. Special
cases of this series have been considered by Petersson for $\tau \in \H$ and $s=k$ in \cite[p 56]{Pe} (also by Zagier  \cite{La} in his proof of the Eichler-Selberg trace formula), and by Cohen for $\tau=0$ and integral $s$ between 2 and $k-2$, see \cite[p 204]{KZ}.
We will see in (\ref{connect}) below that (\ref{omega}) gives a series representation for $\doh_k(z,s)$.

\subsection{Convergence}

\begin{prop} \label{convergence}
The series $\Omega_\ca(z,\tau;s,k)$ defined by (\ref{omega}) is absolutely convergent
\begin{enumerate}
\item for $1<\sigma$ when $\tau \in \H$,
\item for $1<\sigma <k/2$ when $\tau \in \R$,
\item for $1<\sigma < k-1$ when $\sa \tau$ is a cusp of $\G$.
\end{enumerate}
In all cases, the convergence is uniform for $\sigma$ in compact subsets.
\end{prop}
\begin{proof}
First note that $|j(\g,z)|^{-2}=\Im(\g z)/y$. Also
$$
|z^{-s}| \leqslant e^{\pi t}|z|^{-\sigma}
$$
for $z \in \H$  by (\ref{zs}).  Consequently
\begin{equation}\label{rgt}
j(\sb, z)^{-k}\Omega_\ca(\sb z,\tau;s,k) \ll  y^{-k/2} \sum_{\g \in \G} \frac {\Im(\sa^{-1}\g \sb z)^{k/2}} {|\sa^{-1}\g \sb z-\overline{\tau}|^{\sigma}}.
\end{equation}
To estimate the right side of (\ref{rgt}) with an integral, we use the following result
from \cite[(5.2)]{IOS1}, for example. For $h(z)$ holomorphic on $\H$ and $2 \lqs k \in
\R$,
$$
y^{k/2}|h(z)| \leqslant \frac{1}{c_{\varepsilon, k}}
\iint_{\B(z,\varepsilon)} \Im(w)^{k/2} |h(w)| \, d\mu w
$$
with $\B(z,\varepsilon)$ the hyperbolic ball
centered at $z$ of radius $\varepsilon$ and  $c_{\varepsilon, k}$ a constant depending only on
$\varepsilon$ and $k$.  Therefore
\begin{equation}\label{conv}
j(\sb, z)^{-k}\Omega_\ca(\sb z,\tau;s,k) \ll
\frac{y^{-k/2}}{c_{\varepsilon, k}} \sum_{\g \in \G}\iint_{\B(\sa^{-1}\g \sb
z,\varepsilon)} \frac {\Im(w)^{k/2}} {|w-\overline{\tau}|^{\sigma}} \, d\mu w.
\end{equation}
We may choose a radius $\varepsilon$ so that the balls $\B(\sa^{-1}\g \sb z,\varepsilon)$ are
disjoint for all $\g \in \G$, but  $\varepsilon$ will depend on $z$. It is simpler to fix $\varepsilon=1/2$, say, and note that (see \cite[(2.44)]{Iwsp})
$$
\#\{ \g \in \G : \rho(\g z,z)<1\} \ll y_\G(z)+1
$$
where we define   the  {\it invariant height} function
\begin{equation}\label{invht}
y_\G(z):=\max_\ca \max_{\g \in \G}\left(\Im( \sa^{-1}\g z)\right)
\end{equation}
as in \cite[Chapter 2]{Iwsp}. The larger $y_\G(z)$ is, the closer $z$ is to a cusp of $\G$. From \cite[Lemma A.1]{JO2} we have  the upper bound
\begin{equation}\label{ygz}
y_\G(\sb z) \leqslant (c_\G+ 1/c_\G)(y+1/y)
\end{equation}
with any cusp $\cb$. (For a lower bound, consult \cite[Lemma A.2]{JO2}.) Here, $c_\G$ is a positive constant depending only on $\G$ and our choice of inequivalent cusps. (For example, it is  $1/N$ for $\G=\G_0(N)$ with $N$ prime and cusps at $\ci$ and $0$.)
Hence, for all $\g \in \G$,
\begin{equation}\label{c2}
    \Im(\sa^{-1} \g \sb z) \leqslant (c_\G+ 1/c_\G)(y+1/y)
\end{equation}
from (\ref{invht}) and (\ref{ygz}).
Now if $w \in \B(z,1/2)$   then it is easy to verify that $\Im(w) <ey$.  Set
\begin{equation*}
      T(z,\G):=e(c_\G+ 1/c_\G)(y+1/y).
\end{equation*}
We have thus shown that
\begin{equation*}
    \bigcup_{\g \in \G} \B(\sa^{-1} \g \sb
z,1/2) \subseteq B:=\left\{w \in \H : \Im(w) < T(z,\G) \right\}
\end{equation*}
where each point is counted with multiplicity $\ll y+1/y$.
From
(\ref{conv}) we then have
\begin{equation}\label{inte}
j(\sb, z)^{-k}\Omega_\ca(\sb z,\tau;s,k) \ll y^{-k/2}(y+1/y) \iint_{B} \frac {\Im(w)^{k/2}} {|w -\overline{\tau}|^\sigma} \,
d\mu w .
\end{equation}

Let $\alpha+i \beta =-\overline{\tau}$. We consider three
cases.

\vskip 3mm
\noindent
{\bf Case (i).} If $\tau \in \H$ then so is $-\overline{\tau}$ and $\beta>0$. Recall the formula
\begin{equation}\label{integral}
    \int_{-\infty}^\infty \frac{dx}{(x^2+y^2)^{\sigma/2}} = \sqrt{\pi} \frac{\G\bigl((\sigma -1)/2\bigr)}{\G(\sigma/2)}\frac{1}{y^{\sigma -1}}
\end{equation}
for $\sigma >1$. Letting $w=u+iv$ on the right side of (\ref{inte}), we have
\begin{eqnarray*}
  y^{-k/2} (y+1/y)\int_0^{T(z,\G)} \int_{-\infty}^\infty \frac{v^{k/2 -2}}{\left( (\alpha +u)^2 + (\beta +v)^2\right)^{\sigma/2}} \, du dv & \ll & y^{-k/2} (y+1/y)\int_0^{T(z,\G)} \frac{v^{k/2 -2}}{(\beta +v)^{\sigma -1}} \,  dv \\
   & \ll & y^{-k/2} (y+1/y)\int_1^{T(z,\G)} v^{k/2 -1-\sigma} \,  dv \\
   & \ll & y^{-k/2}(y+1/y)^{k/2 -\sigma+1}
\end{eqnarray*}
provided $1<\sigma$ and $k>2$. We have arrived at the bound
\begin{equation}\label{case1}
    j(\sb, z)^{-k}\Omega_\ca(\sb z, \tau;s,k) \ll y^{1-\sigma}+y^{\sigma-k-1}
\end{equation}
with an implied constant depending only on $\tau, s, k$ and $\G$. Therefore, for $\tau \in \H$, (\ref{omega})
is absolutely convergent for $1<\sigma$. The convergence is uniform for $\sigma$ in compact sets.

\vskip 3mm
\noindent
{\bf Case (ii).} If $\tau  \in \R$ then $\beta =0$. A similar analysis to Case (i) above shows
(\ref{case1}) also holds
for $1< \sigma < k/2$ with an implied constant depending only on $\tau, s, k$ and $\G$.
Therefore, for $\tau  \in \R$, (\ref{omega})
is absolutely convergent for $1< \sigma < k/2$. The convergence is uniform for $\sigma$ in compact sets.

\vskip 3mm
\noindent
{\bf Case (iii).} If $\sa\tau$ is a cusp of $\G$ then we may write $\sa\tau=\delta \cc$ for some $\delta \in \G$ and $\cc$ one of our set of inequivalent cusps. We need  the following lemma which essentially says that, for each cusp of $\G$, points in a $\G$-orbit of $z$ must lie outside some disc in $\H$ that is tangent to $\R$ at that cusp.

\begin{lemma}\label{tech}
For all $\g \in \G$, $z \in \H$ and $\tau\in \R$ with $\sa\tau$ a cusp  of $\G$, we have
\begin{equation}\label{c3}
\frac{\Im( \sa^{-1} \g \sb z) }{ \left|   \sa^{-1} \g \sb z  - \tau \right|^2} \ll y+1/y
\end{equation}
where the implied constant depends on $\G$ and $\tau$ alone.
\end{lemma}
\begin{proof}
We have  $\delta^{-1}  \g \in \G$ and by (\ref{c2})
\begin{equation}\label{im1}
\Im(\sc^{-1} (\delta^{-1} \g) \sb z)  \ll (y+1/y)
\end{equation}
with an implied constant depending only on $\G$. Also
\begin{eqnarray}
  \Im(\sc^{-1} \delta^{-1}  \g \sb z) &=& \Im(\sc^{-1} \delta^{-1} \sa \cdot \sa^{-1} \g \sb z) \nonumber \\
   &=&  \Im(\left(\smallmatrix * & * \\ c & d \endsmallmatrix \right)  \sa^{-1} \g \sb z) \label{im2}
\end{eqnarray}
on labelling $\sc^{-1} \delta^{-1} \sa$ as $\left(\smallmatrix * & * \\ c & d \endsmallmatrix \right)$. We have
$$
\left(\smallmatrix d & * \\ -c & * \endsmallmatrix \right) = (\sc^{-1} \delta^{-1} \sa)^{-1} = \sa^{-1} \delta \sc
$$
so that
\begin{equation}\label{dc}
\tau = \sa^{-1} \delta \cc = \sa^{-1} \delta \sc \ci = -d/c.
\end{equation}
Since $\tau \in \R$, (\ref{dc}) implies $c \neq 0$. Hence
\begin{eqnarray}
  \Im(\left(\smallmatrix * & * \\ c & d \endsmallmatrix \right)  \sa^{-1} \g \sb z)  &=& \frac{\Im(  \sa^{-1} \g \sb z) }{\left|c \cdot \sa^{-1} \g \sb z +d \right|^2} \nonumber \\
   &=& \frac{\Im( \sa^{-1} \g \sb z) }{\left|c\right|^2 \left|   \sa^{-1} \g \sb z  - \tau \right|^2}. \label{im3}
\end{eqnarray}
It follows from (\ref{im1}), (\ref{im2}) and (\ref{im3}) that
$$
\frac{\Im( \sa^{-1} \g \sb z) }{\left|c\right|^2 \left|   \sa^{-1} \g \sb z  - \tau \right|^2} \ll y+1/y
$$
and, since $c$ depends only on $\G$ and the choice of cusps and $\tau$, the proof of the lemma is complete.
\end{proof}

Now  $z' \in \B(z,\varepsilon)$ if and only if $\sa^{-1} \g \sb z' \in \B(\sa^{-1} \g \sb z,\varepsilon)$.
Also, for $z=x+iy$, $z'=x'+iy'$ we see  that $z' \in \B(z,\varepsilon)$ implies $y'+1/y' < e(y+1/y)$ for $\varepsilon <1$. Thus, if we replace $z$ in (\ref{c3}) by $z'\in \B(z,\varepsilon)$ we find
$$
\frac{\Im( \sa^{-1} \g \sb z') }{ \left|   \sa^{-1} \g \sb z'  - \tau \right|^2} \ll y+1/y.
$$
Hence $w \in \B(\sa^{-1} \g \sb z,\varepsilon)$ implies
\begin{equation}\label{bprime}
\frac{1}{|w-\tau|^2} \ll \frac{y+1/y}{\Im(w)}
\end{equation}
for an implied constant depending only on $\G$ and $\tau$. Let $B'$ be the elements $w$ of $B$ that also satisfy (\ref{bprime}). For $1< \sigma$ choose $r$ satisfying $1<r<\sigma$. From (\ref{inte}) we obtain
\begin{eqnarray*}
j(\sb, z)^{-k}\Omega_\ca(\sb z, \tau;s,k) & \ll & y^{-k/2} (y+1/y)\iint_{B'} \frac {\Im(w)^{k/2}} {|w-\tau|^\sigma} \, d\mu w  \\
   &=& y^{-k/2} (y+1/y)\iint_{B'} \frac {\Im(w)^{k/2}} {|w-\tau|^r |w-\tau|^{\sigma-r}} \, d\mu w  \\
   & \ll & y^{-k/2} (y+1/y)\int_0^{T(z,\G)} \int_{-\infty}^\infty \frac{v^{k/2 -2 - (\sigma-r)/2} (y+1/y)^{(\sigma-r)/2}}{\left( (\alpha +u)^2 + v^2\right)^{r/2}} \, du dv \\
   & \ll & y^{-k/2} (y+1/y)^{(\sigma-r)/2+1} \int_0^{T(z,\G)} v^{(k-\sigma -r)/2 -1} \, dv.
\end{eqnarray*}
Thus, for $\sigma< k-r$ we conclude that
\begin{equation}\label{case3}
    j(\sb, z)^{-k}\Omega_\ca(\sb z, \tau;s,k) \ll y^{1-r}+y^{r-k-1}.
\end{equation}
Therefore, for $\sa\tau$ a cusp, (\ref{omega})
is absolutely convergence for $1< \sigma < k-1$. The convergence is uniform for $\sigma$ in compact sets.
\end{proof}

\begin{prop} \label{gencusp}
The series $\Omega_\ca(z,\tau;s,k)$, as a function of $z$, is
 in $S_k(\G)$  when $s$ is in the  domain of absolute convergence corresponding to $\tau$ described in Proposition \ref{convergence} (and $\sigma<k+1$ in case (i)).
\end{prop}
\begin{proof}
It is clear that $\Omega_\ca(z,\tau;s,k)$ is a holomorphic function of $z$, since the convergence in Proposition \ref{convergence} is uniform. That it is weight $k$ in $z$ is easily verified. It only remains to check that it decays as $z$ approaches each cusp $\cb$. To this end we consider
\begin{equation}\label{fab}
  j(\sb, z)^{-k}\Omega_\ca(\sb z,\tau;s,k).
\end{equation}
Verifying that (\ref{fab}) is invariant as $z \to z+1$, it must have the  Fourier expansion
\begin{equation*}
  j(\sb, z)^{-k}\Omega_\ca(\sb z, \tau;s,k)=\sum_{m \in \Z} a_{\ca \cb}(m) e^{2 \pi i m z}.
\end{equation*}
With (\ref{case1}) and (\ref{case3}) we find
$$
a_{\ca \cb}(m) =\int_{iY}^{1+iY} j(\sb, z)^{-k}\Omega_\ca(\sb z, \tau;s,k) e^{-2\pi i m z} \, dz \ll Y^{-A} e^{2\pi m Y}
$$
for some $A>0$. Thus, letting $Y \to \infty$, we see that $a_{\ca \cb}(m) =0$ for $m \leqslant 0$. This completes the proof that $\Omega_\ca(z,\tau;s,k)$ is a cusp form.
\end{proof}

\subsection{Analytic Continuation}
\subsubsection{Continuation to a left half-plane} \label{abc}
We review next some results that we will require on the symmetrized
Hurwitz zeta function
\begin{equation}\label{hurw2}
    \zeta_\Z(z,s):=\sum_{n \in \Z} \frac{1}{(z+n)^s}.
\end{equation}
We shall only be concerned with $z \in \H$. Clearly (\ref{hurw2}) is absolutely convergent for $\Re(s)>1$.  Using Lipschitz
summation we have, as in \cite[Thm. 1]{KR},
\begin{equation}\label{lip}
\zeta_\Z(z,s)=\frac{(2\pi)^s }{e^{s i \pi/2} \G(s)}
\sum_{n=1}^\infty n^{s-1} e^{2 \pi i n z}.
\end{equation}
It is clear that (\ref{lip}) is now an analytic function of $s$ for  all $s \in \C$, extending the definition of   $\zeta_\Z(z,s)$.
Moreover, with elementary estimates on (\ref{lip}) we obtain
\begin{equation}\label{lip2}
\zeta_\Z(z,s) \ll \begin{cases} e^{-2 \pi
y}\left(1+y^{-\sigma}\right) & \text{ if \ \ } \sigma \neq 0  \\
e^{-2 \pi
y}\left(1+|\log y |\right) & \text{ if \ \ } \sigma =0,
\end{cases}
\end{equation}
for all $s \in \C$, $z \in \H$ and an implied constant depending only on $s$.

\vskip 3mm
Rearranging the absolutely convergent
(\ref{omega}), we have
\begin{eqnarray}
\Omega_\ca(z,\tau;s,k) & = & \sum_{\g \in \G_\ca \backslash \G} \sum_{n \in \Z} \frac 1 {(\sa^{-1} \g z+n-\overline{\tau})^{s} j(\sa^{-1}\g,z)^k} \nonumber\\
   &=& \sum_{\g \in \G_\ca \backslash \G}  \frac{\zeta_{\Z}(\sa^{-1}\g z-\overline{\tau}, s)}{j(\sa^{-1}\g,z)^k} \label{hz}
\end{eqnarray}
for all $s$ with $1 < \sigma <k/2$. But from (\ref{lip2}) we have
\begin{eqnarray*}
\sum_{\g \in \G_\ca \backslash \G}  \frac{\zeta_{\Z}(\sa^{-1}\g z -\overline{\tau}, s)}{j(\sa^{-1}\g,z)^k} & \ll & y^{-k/2}\sum_{\g \in \G_\ca \backslash \G} \left(\Im(\sa^{-1}\g z)^{k/2 - \sigma}+ \Im(\sa^{-1}\g z)^{k/2}\right)\\
   &=&  y^{-k/2} \left( E_\ca(z, k/2 - \sigma) + E_\ca(z,k/2) \right)
\end{eqnarray*}
provided $k/2 - \sigma>1$ and $k/2 >1$. The Eisenstein series $E_\ca(z,s)$ are absolutely convergent for $\sigma>1$, as in \cite{Iwsp}.  Thus we see that the right hand side of (\ref{hz}) converges to an analytic function of $s$ for $\sigma< k/2-1$. If $k \gqs 6$ then $\sigma< k/2-1$ overlaps with $1<\sigma$ and we have shown that $\Omega_\ca(z,\tau;s,k)$ has an analytic continuation to a left half plane.

\subsubsection{Continuation to all of $\C$}

Now let $f(z) \in S_k(\G)$ have Fourier expansion at the cusp $\ca$
$$
j(\sa,z)^{-k} f(\sa z) = \sum_{m=1}^\infty
a_\ca(m) e^{2\pi i m z}
$$ and for $\mu \in \H \cup \R$ define
$$
L_\ca(f,s;\mu):=\sum_{m=1}^\infty \frac{a_\ca(m) e^{2\pi i m \mu}}{m^{s}}.
$$
From the  bound $a_\ca(m) \ll m^{(k-1)/2}$ on average, as in \cite[Corollary
5.2]{Iwto},
we see that $L_\ca(f,s;\mu)$ is absolutely convergent for $\Re(s)>(k+1)/2$.
Also
\begin{equation}\label{star}
L^*_\ca(f,s;\mu):=(2\pi)^{-s} \G(s) L_\ca(f,s;\mu) = \int_0^\infty (f|_k \sa)(iy+\mu) y^{s-1} \, dy
\end{equation}
which is analytic for $s$ in all of $\C$ for $\Im(\mu)>0$ or $\sa \mu$ a cusp of $\G$.

\begin{prop} \label{inn}
For $1< \sigma < k/2-1$ we have
$$
\s{\Omega_\ca( \cdot ,\tau;s,k)}{f}=2^{2-k} \pi e^{-s i \pi/2}  \frac{\G(k-1)}{\G(s)\G(k-s)} L^*_\ca(\overline{f},k-s;-\overline{\tau}).
$$
\end{prop}
\begin{proof}
After unfolding we have
\begin{eqnarray*}
  \s{\Omega_\ca( \cdot ,\tau;s,k)}{f} &=& \int_0^\infty \int_0^1  y^{k-2}   \zeta_\Z (z-\overline{\tau},s) \overline{j(\sa,z)}^{-k} \overline{f(\sa z)} \, dx dy \\
   &=& \frac{(2\pi)^s }{e^{s i \pi/2} \G(s)} \int_0^\infty \int_0^1
   y^{k-2} \left( \sum_{m=1}^\infty
m^{s-1} e^{2 \pi i m (z-\overline{\tau})} \right)\left( \sum_{n=1}^\infty \overline{a_\ca(n)} e^{-2 \pi i n\overline z}
\right) \, dx dy \\
   &=& \frac{(2\pi)^s }{e^{s i \pi/2} \G(s)} \int_0^\infty
  y^{k-2} \left( \sum_{m=1}^\infty
 m^{ s-1} \overline{a_\ca(m)} e^{-4 \pi  m y} e^{-2\pi i m \overline{\tau}} \right) \, dy \\
   &=& \frac{(2\pi)^s }{e^{s i \pi/2} \G(s)} \frac{\G(k-1)}{(4 \pi)^{k-1}} L_\ca(\overline{f},k-s;-\overline{\tau})
\end{eqnarray*}
and (\ref{star}) completes the proof.
\end{proof}

 Let $f_j$ with $1 \leqslant j \leqslant n$ be an
orthonormal basis for $S_k(\G)$. We find
\begin{eqnarray}
  \Omega_\ca( z ,\tau;s,k) &=& \sum_j \s{\Omega_\ca( \cdot ,\tau;s,k)}{f_j} f_j(z) \nonumber \\
   &=& 2^{2-k} \pi e^{-s i \pi/2}  \frac{\G(k-1)}{\G(s)\G(k-s)} \sum_j L^*_\ca(\overline{f_j},k-s;-\overline{\tau})
   f_j(z). \label{cont2}
\end{eqnarray}
Thus (\ref{cont2}) gives  the continuation of
$\Omega_\ca( z ,\tau;s,k)$ to all $s \in \C$, except in the case when $\tau \in \R$ and $\sa \tau$ is not a cusp. We have proved

\begin{theorem} \label{nine}
Let $k \gqs 6$. The series $\Omega_\ca( z ,\tau;s,k)$, originally defined by (\ref{omega}) for at least $1<\Re(s)<k/2$, has a meromorphic continuation to all $s \in \C$ for  $\tau \in \H$ or  $\tau \in \R$ and $\sa \tau$ a cusp of $\G$. In the case that $\tau \in \R$ and $\sa \tau$ is not a cusp we only have the continuation for $\Re(s)<k/2$. In all cases $\Omega_\ca( z ,\tau;s,k)$ is a cusp form in $z$.
\end{theorem}

\noindent
{\bf Remark.}
If we set $s=k$ in (\ref{omega}) and restrict to $\tau \in \H$ we obtain a series first considered by Petersson \cite{Pe}.
We have $\Omega_\ci(z,\tau;k,k)$ in $S_k(\G)$ both as a function of $z$ and of $\tau$.
By Proposition \ref{inn}
$$
\s{\Omega_\ci(\cdot ,\tau;k,k)}{f} =  (-1)^{k/2}2^{2-k} \pi /(k-1) \overline{f(\tau)}.
$$
so that $\Omega_\ci(z,\tau;k,k)$ is a reproducing kernel.
Hence
$$
\Omega_\ci(z ,\tau;k,k) = (-1)^{k/2}2^{2-k} \pi /(k-1) \sum_{j} f_j(z) \overline{f_j(\tau)}
$$
for any orthonormal basis $f_j$. The series $\Omega_\ci(z ,\tau;k,k)$ my also be recognized as the $0$th elliptic Poincar\'e series, see \cite{Pe}, \cite[p 260]{EG} or \cite[Section 4]{IOS1}.

\subsection{Cohen's kernel at general arguments}
We examine in more detail some special cases of Theorem \ref{nine}, including the connection with $\doh_k(z,s)$.
Let $\G=\G_0(N)$, the  Hecke congruence group of level $N$ for the remainder of this section. It has cusps at $\ci$ and $0$. These are $\G$-equivalent when $N=1$ and inequivalent otherwise. In either case set
$$
\coh_k(z,s) := \Omega_\ci(z, 0;s,k).
$$
With Theorem \ref{nine} we see that it is a cusp form for all $s\in \C$. For any cusp form $f$ we have
\begin{equation}\label{cop1}
    \s{\coh_k(\cdot,s)}{f}=2^{2-k} \pi e^{-s i \pi/2}  \frac{\G(k-1)}{\G(s)\G(k-s)} L^*(\overline{f},k-s).
\end{equation}
Comparing (\ref{cop1}) with (\ref{doh}) and using (\ref{feqqq9}) yields
\begin{equation}\label{connect}
\coh_k(z,s) = 2^{2-k}(-1)^{k/2} \pi e^{-s i \pi/2}  \frac{\G(k-1)}{\G(s)\G(k-s)} \doh_k(z,s)
\end{equation}
for $\G=\G_0(1)$.

\subsection{A functional equation} \label{funeq9}
 Let $\omega=\left(\begin{matrix} 0 & -1 \\ N & 0
\end{matrix}\right)$. We have $\G_0(N)= \omega^{-1} \G_0(N) \omega$.
As in \cite[p112]{Iwto} define the operator $W:S_k(\G_0(N)) \to
S_k(\G_0(N))$ by $W f = f|_k \omega$. We have $W^2f=f$ and $\s{Wf}{Wg}
= \s{f}{g}$ for all $f$, $g$ in $S_k(\G)$. Therefore $W$ is
self-adjoint and we may choose our orthonormal basis to be
eigenfunctions of $W$:
$$
Wf_j=\eta_j f_j
$$
for all $j$ with $\eta_j=\pm 1$ necessarily.

\begin{theorem} \label{fun}
For all $s \in \C$
\begin{equation}\label{fe2}
\coh_k(z,k-s) = e^{s i \pi} N^{k/2-s}  \coh_k( z, s)|_k \omega.
\end{equation}
\end{theorem}
\begin{proof}
Starting with the
original definition (\ref{cohen}) of $\coh_k(z,s)$, which we know is absolutely
convergent for $1 < \Re(s) <k-1$,
\begin{eqnarray*}
  \coh_k(z,s) &=& \sum_{\g \in \G} \frac{1}{(\omega^{-1} \g \omega z)^s j(
\omega^{-1} \g \omega,z)^k} \\
   &=& \sum_{\g \in \G} \frac{1}{\left(\frac{-1}{ N\g \omega z}\right)^s j( \omega^{-1},
\g \omega z)^k
   j(  \g ,\omega z)^k j( \omega,z)^k} \\
   &=& \frac{N^s e^{-s i \pi}}{j(\omega , z)^k} \sum_{\g \in \G}
   \frac{1}{(\g \omega z )^{k-s}
   j(  \g ,\omega z)^k} \\
   &=& N^{s-k/2} e^{-s i \pi} \coh_k( z,
   k-s)|_k \omega
\end{eqnarray*}
where we used
$$
\left(-1/z\right)^s = e^{s \log (-1/z)} = e^{s(i\pi -\log z)}=e^{s i\pi} \cdot e^{-s \log z} = e^{s i\pi} z^{-s}
$$
for $z \in \H$. The proof follows by analytic continuation.
\end{proof}

The functional equations for the $L$-functions $L(f_j,s)$ may be recovered easily from Theorem \ref{fun}. Simply apply (\ref{fe2}) to (\ref{cont2}) and equate coefficients of $f_j$ to get
\begin{equation}\label{lfun}
\left(\frac{\sqrt{N}}{2 \pi} \right)^s \G(s) L(f_j,s) = \eta_j  \ i^k
\left(\frac{\sqrt{N}}{2 \pi} \right)^{k-s} \G(k-s) L(f_j,k-s).
\end{equation}
See \cite[Thm.
7.2]{Iwto}, for example, for the standard proof of  (\ref{lfun}).

\bibliography{kernel}

\end{document}